\documentclass[letterpaper]{article}
\usepackage{graphicx,xcolor}
\usepackage{amsmath,amsthm}
\usepackage{amssymb}
\usepackage{float}

\usepackage[english]{babel}

\theoremstyle{plain}

\newtheorem*{theorem*}{Theorem}
\newtheorem{theorem}{Theorem}[section]
\newtheorem{claim}{Claim}[section]

\newtheorem{lemma}[theorem]{Lemma}
\newtheorem{proposition}[theorem]{Proposition}
\newtheorem{corollary}[theorem]{Corollary}

\newtheorem*{questions*}{Questions}
\newtheorem*{example*}{Example}

\def\NN{\mathbb{N}}
\def\ZZ{\mathbb{Z}}

\def\ag{\mathcal{A}}

\def\CC{\mathcal{C}}

\def\FF{\mathcal{F}}
\def\Orb{\text{Orb}}
\def\Aut{\text{Aut}}

\newcommand{\define}[1]{\emph{#1}}

\usepackage{tikz}
\usetikzlibrary{patterns,positioning,arrows,decorations.markings,calc,decorations.pathmorphing,decorations.pathreplacing}
\usetikzlibrary{lindenmayersystems}
\usetikzlibrary{arrows}
\usepackage{tikz-3dplot}

\tikzstyle directedgreen=[postaction={decorate,decoration={markings,
    mark=at position .65 with {\arrow[scale=1]{latex}}}}]
\tikzstyle directedred=[postaction={decorate,decoration={markings,
    mark=at position .55 with {\arrow[scale=1]{latex}}}}]

\tdplotsetmaincoords{85}{97}
  
\definecolor{rouge}{RGB}{255,77,77}
\definecolor{vert}{RGB}{0,178,102}
\definecolor{jaune}{RGB}{255,255,0}
\definecolor{violet}{RGB}{208,32,144}
\definecolor{orange}{RGB}{255,140,0}
\definecolor{bleu}{RGB}{0,0,205}

\newcommand{\tilea}{%
	\begin{tikzpicture}[scale = 0.6]
	\draw [ thick] (-0.2,-0.2) rectangle (0.2,0.2);
	\end{tikzpicture}
}
\newcommand{\tileb}{%
	\begin{tikzpicture}[scale = 0.6]
	\draw [ thick, fill = black] (-0.2,-0.2) rectangle (0.2,0.2);
	\end{tikzpicture}
}

\title{A generalization of the simulation theorem for semidirect products\\}
\date{}
\author{Sebasti\'an Barbieri and Mathieu Sablik}

\begin{document}
 
\maketitle 
 
\begin{abstract}
We generalize a result of Hochman in two simultaneous directions: Instead of realizing an effectively closed $\ZZ^d$ action as a factor of a subaction of a $\ZZ^{d+2}$-SFT we realize an action of a finitely generated group analogously in any semidirect product of the group with $\ZZ^2$. Let $H$ be a finitely generated group and $G = \ZZ^2 \rtimes H$ a semidirect product. We show that for any effectively closed $H$-dynamical system $(Y,f)$ where $Y$ is a Cantor set, there exists a $G$-subshift of finite type $(X,\sigma)$ such that the $H$-subaction of $(X,\sigma)$ is an extension of $(Y,f)$. In the case where $f$ is an expansive action of a recursively presented group $H$, a subshift conjugated to $(Y,f)$ can be obtained as the $H$-projective subdynamics of a $G$-sofic subshift. As a corollary, we obtain that $G$ admits a non-empty strongly aperiodic subshift of finite type whenever the word problem of $H$ is decidable.
\end{abstract}

\section{Introduction}

A dynamical system is a tuple $(X,T)$ where $X$ is a set and $T: X \to X$ is a map which describes the evolution of points of $X$ in time. In the case where $T$ is bijective one can describe $T$ as a $\ZZ$-action by associating $(n,x) \to T^n(x)$. This can be generalized to a set of bijective maps $T_1,\dots,T_n$ which satisfy some set of relations $R$ --for instance, the relation $T_1\circ T_2 = T_2 \circ T_1$ which indicates $T_1$ and $T_2$ commute--. These actions and their relations can be expressed by the group action $\mathcal{T}: G \times X \to X$ where $G \cong \langle T_1,\dots,T_n \mid R \rangle$ and $\mathcal{T}(T_{i_1} \circ \dots \circ T_{i_k},x) = T_{i_1} \circ \dots \circ T_{i_k}(x)$.

More than often dynamical systems arising from group actions are difficult to study, and a fruitful technique is to look at their subactions, that is, the restriction of the group action to a particular subgroup. For instance, see the study of expansive subdynamics of $\ZZ^d$ actions~\cite{BL1997expansive, einsiedler2001expansive}. It is thus appealing to ask the following question: What systems can be obtained as subactions of a class of dynamical systems? An interesting class is the one of subshifts of finite type (SFT), that is, the sets of colorings of a group along with the shift action which are defined by a finite number of forbidden patterns.

For the class of $\ZZ^d$-SFTs there is still no characterization of which dynamical systems can arise as their subactions, nevertheless, it has been proven by Hochman~\cite{Hochman2009b} that every $\ZZ^d$-action over a cantor set $T: \ZZ^d \times X \to X$ which is effectively closed -- meaning that it can be described with a Turing machine-- admits an almost trivial isometric extension which can be realized as the subaction of a $\ZZ^{d+2}$-SFT. This result has subsequently been improved for the expansive case independently in \cite{AubrunSablik2010} and \cite{DurandRomashchenkoShen2010} showing that every effectively closed subshift can be obtained as the projective subdynamics of a sofic $\ZZ^2$-subshift. These kind of results yield powerful techniques to prove properties about the original systems. An example is the characterization of the set of entropies of $\ZZ^2$-SFTs~\cite{HochmanMeyerovitch2010} as the set of right recursively enumerable numbers.

In this article we extend Hochman's result to the case of group actions for groups which are of the form $G = \ZZ^2 \rtimes_{\varphi} H$ for some finitely generated group $H$ and an homomorphism $\varphi: H \to \Aut(\ZZ^2)$. More specifically we prove the following result.
{
	\renewcommand{\thetheorem}{\ref{simulationTHEOREM}}
	\begin{theorem}
		For every $H$-effectively closed dynamical system $(X,f)$ there exists a $(\ZZ^2 \rtimes H)$-SFT whose $H$-subaction is an extension of $(X,f)$.
	\end{theorem}
	\addtocounter{theorem}{-1}
}

We remark the strong gap which occurs when passing from $\ZZ$-SFTs to the multidimensional case. For instance, $\ZZ$-SFTs contain periodic points, have regular languages and the possible set of entropies they can have is reduced to logarithms of Perron numbers~\cite{lind1995introduction}. In the other hand multidimensional SFTs can be strongly aperiodic~\cite{Berger1966,Robinson1971,Kari1996259,JeandelRao11Wang}, can be composed uniquely of non-computable points~\cite{Hanf1974,Myers1974} and their entropies are not even computable~\cite{HochmanMeyerovitch2010}. Most of these differences can be put into evidence with simulation theorems by the fact that multidimensional SFTs can be projected onto effectively closed subshifts in one dimension. Our Theorem~\ref{simulationTHEOREM} allows analogously to extend properties of effectively closed subshifts in general groups $H$ and show that they also appear in SFTs when the group is replaced by $\ZZ^2 \rtimes H$. This is a powerful tool to construct examples of groups with admit subshifts of finite type with some desired property which is easier to realize in an effective subshift.

Readers who are not familiar with computability or the embedding of Turing machine computations in subshifts of finite type will be reassured by the fact that in the proof all of those aspects are hidden in black boxes. Namely, we use the result of~\cite{AubrunSablik2010,DurandRomashchenkoShen2010} that every effectively closed $\ZZ$-subshift is the projective subdynamics of a sofic $\ZZ^2$-subshift whose vertical shift action is trivial. We also make use of a theorem of Mozes~\cite{mozes} which states that subshifts arising from two-dimensional substitutions are sofic.

In the case when $H$ is a recursively presented group, Theorem~\ref{simulationTHEOREM} can be presented in a purely symbolic dynamics fashion for expansive actions, namely we show:

{
	\renewcommand{\thetheorem}{\ref{theorem.projective}}
	\begin{theorem}
		Let $X$ be an effectively closed $H$-subshift. Then there exists a sofic $(\ZZ^2 \rtimes H)$-subshift $Y$ such that its $H$-projective subdynamics $\pi_H(Y)$ is $X$.
	\end{theorem}
	\addtocounter{theorem}{-1}
}

It is known that every $\ZZ$-SFT contains a periodic configuration~\cite{lind1995introduction}. However, it was shown by Berger~\cite{Berger1966} that there are $\ZZ^2$-SFTs which are strongly aperiodic, that is, such that the shift acts freely on the set of configurations. This result has been proven several times with different techniques~\cite{Robinson1971,Kari1996259,JeandelRao11Wang} giving a variety of constructions. However, it remains an open question which is the class of groups which admit strongly aperiodic SFTs. Amongst the class of groups that do admit strongly aperiodic SFTs are: $\ZZ^d$ for $d > 1$, hyperbolic surface groups~\cite{CohenGoodmanS2015}, Osin and Ivanov monster groups~\cite{Jeandel2015}, and the direct product $G \times \ZZ$ for a particular class of groups $G$ which includes Thompson's $T$ group and $\text{PSL}(\ZZ,2)$~\cite{Jeandel2015}. It is also known that no group with two or more ends can contain strongly aperiodic SFTs~\cite{Cohen2014} and that recursively presented groups which admit strongly aperiodic SFTs must have decidable word problem~\cite{Jeandel2015}. 

As an application of Theorem~\ref{simulationTHEOREM} we present a new class of groups which admit strongly aperiodic SFTs, that is:

{
	\renewcommand{\thetheorem}{\ref{theoremaperiodic}}
	\begin{theorem}
		Every semidirect product $\ZZ^2 \rtimes H$ where $H$ is finitely generated and has decidable word problem admits a non-empty strongly aperiodic SFT.
	\end{theorem}
	\addtocounter{theorem}{-1}
}

Amongst this new class of groups which admit strongly aperiodic SFTs, we remark the Heisenberg group which admits a presentation $\mathcal{H} \cong \ZZ^2 \rtimes \ZZ$.

\section{Preliminaries}

Consider a group $G$ and a compact topological space $(X,\mathcal{T})$. The tuple $(X,f)$ where $f: G \times X \to X$ is a left $G$ action by homeomorphisms is called a $G$-flow (or $G$-dynamical system). Let $(X,f)$, $(X',f')$ be two $G$-flows. We say $\phi: X \to X'$ is a \define{morphism} if it is continuous and $\phi \circ f_g = f'_g \circ \phi$ for all $g \in G$. A surjective morphism $\phi : X \twoheadrightarrow X'$ is a \define{factor} and we say that  $(X',f')$ is a \define{factor} of $(X,f)$ and that $(X,f)$ is an \define{extension} of $(X',f')$. When $\phi$ is a bijection and its inverse is continuous we say it is a \define{conjugacy} and that $(X,f)$ is \define{conjugated} to $(X',f')$.

%
%

In what follows, we consider only cantor sets with the product topology and finitely generated groups. Without loss of generality, we consider actions over closed subsets of $\{0,1\}^{\NN}$. Let $G$ be a group generated by a finite set $S$. A \define{$G$-effectively closed flow} is a $G$-flow $(X,f)$ where:

\begin{enumerate}
	\item $X \subset \{0,1\}^{\NN}$ is a closed effective subset: $X = \{0,1\}^{\NN} \setminus \bigcup_{i \in I}{[w_i]}$ where $\{w_i\}_{i \in I} \subset \{0,1\}^*$ is a recursively enumerable language. That means that $X$ is the complement of a union of cylinders which can be enumerated by a Turing machine.
	\item $f$ is an effectively closed action: there exists a Turing machine which on entry $s \in S$ and $w \in \{0,1\}^*$ enumerates a sequence of words $(w_j)_{j \in J}$ such that $f_{s}^{-1}([w]) = \{0,1\}^{\NN} \setminus \bigcup_{j \in J}[w_j]$.
\end{enumerate}

The idea behind the definition is the following: There is a Turing machine $T$ which given a word $g \in S^*$ representing an element of $G$ and $n$ coordinates of $x \in X \subset \{0,1\}^{\NN}$ returns an approximation of the preimage of $x$ by $f_g$.

Let $\ag$ be a finite alphabet and $G$ a finitely generated group. The set $\ag^G = \{ x: G \to \ag\}$ equipped with the left group action $\sigma: G \times \ag^G \to \ag^G$ given by: 
$(\sigma_{h}(x))_g = x_{h^{-1}g}$ is the \textit{$G$-full shift}. The elements $a \in \ag$ and $x \in \ag^G$ are called \define{symbols} and \define{configurations} respectively. We endow $\ag^G$ with the product topology, therefore obtaining a compact metric space. The topology is generated by the metric $\displaystyle{d(x,y) = 2^{-\inf\{|g|\; \mid\; g \in G:\; x_g \neq y_g\}}}$ where $|g|$ is the length of the smallest expression of $g$ as the product of some fixed set of generators. This topology is also generated by a clopen basis given by the \define{cylinders} $[a]_g = \{x \in \ag^G | x_g = a\in \ag\}$. A \emph{support} is a finite subset $F \subset G$. Given a support $F$, a \emph{pattern with support $F$} is an element $P$ of $\ag^F$, i.e. a finite configuration and we write $supp(P) = F$. We also denote the cylinder generated by $P$ centered in $g$ as $[P]_g = \bigcap_{h \in F}[P_h]_{gh}$. If $x \in [P]_g$ for some $g \in G$ we write $P \sqsubset x$.

A subset $X$ of $\ag^G$ is a \define{$G$-subshift} if it is $\sigma$-invariant -- $\sigma(X)\subset X$ -- and closed for the cylinder topology. Equivalently, $X$ is a $G$-subshift if and only if there exists a set of forbidden patterns $\FF$ that defines it.
$$X=X_\FF :=  {\ag^G \setminus \bigcup_{P \in \FF, g \in G} [P]_g}.$$

That is, a $G$-subshift is a subset of $\ag^G$ which can be written as the complement of a union of cylinders.

If the context is clear enough, we will drop the $G$ and simply refer to a subshift. A subshift $X\subseteq \ag^G$ is \define{of finite type} -- SFT for short -- if there exists a finite set of forbidden patterns $\FF$ such that $X=X_\FF$. A subshift $X\subseteq \ag^G$ is \define{sofic} if there exists a subshift of finite type $Y \subset \ag'^G$ and a factor $\phi: Y \twoheadrightarrow X$. A subshift is effectively closed if there exists a recursively enumerable coding of a set of forbidden patterns $\FF$ such that $X=X_\FF$. More details can be found in~\cite{ABS2014} or in Section~\ref{section.consequences}. 

Any $G$-flow over a cantor set can be seen as a subshift over an infinite alphabet: Indeed, $(X,f)$ can be seen as $Y \subset (\{0,1\}^{\NN})^G$ equipped with the shift action such that $x \in Y$ if and only if $\forall g \in G$ $x_g = f_g(x_{1_G})$. In this setting, effectively closed $G$-flows correspond to effectively closed subshifts in this infinite alphabet.

Let $H \leq G$ be a subgroup and $(X,f)$ a $G$-flow. The \emph{$H$-subaction} of $(X,f)$ is $(X,f^H)$ where $f^H : H \times X \to X$ is the restriction of $f$ to $H$, that is $\forall h \in H, (f^H)_h(x) = f_h(x)$. In the case of a subshift $X \subset \ag^G$ there is also the different notion of projective subdynamics. The \emph{$H$-projective subdynamics} of $X$ is the set $\pi_H(X) = \{ y \in \ag^H \mid \exists x \in X, \forall h \in H, y_h = x_h \}$. 

\section{Simulation Theorem}
\label{section.generalities}

The purpose of this section is to prove our main result.

\begin{theorem}\label{simulationTHEOREM}
	Let $H$ be finitely generated group and $G = \ZZ^2 \rtimes H$. For every $H$-effectively closed flow $(X,f)$ there exists a $G$-SFT whose $H$-subaction is an extension of $(X,f)$.
\end{theorem}

We begin by introducing some general constructions. The general schema of the proof is the following: First we construct for each non-zero vector $v \in (\ZZ/3\ZZ)^2$ a substitution $\texttt{s}_v$ which encodes countable copies of $\ZZ^2$ as lattices with the property that any automorphism $\varphi \in \Aut(\ZZ^2)$ sends each of the lattices of $\texttt{s}_v$ to those of $\texttt{s}_{\tilde{\varphi}(v)}$ where $\tilde{\varphi} \in \Aut((\ZZ/3\ZZ)^2)$ is the automorphism obtained by projecting each component to $\ZZ/3\ZZ$. This structure allows us to pair lattices of $\ZZ^2$ when moving in $G$ by elements of $H$.

Then we encode the elements of $X$ and the $H$-flow $f$ in an effective Toeplitz $\ZZ$-subshift. We do so in a way that the projections of the $n$-th order lattice to the line in the previous construction always matches with the symbol $x_n$. For technical reasons of matching all possible lattices, we do this coding in two different ways.

Afterwards, we extend the Toeplitz subshift to a $\ZZ^2$-subshift by repeating rows (or columns). Using a known simulation theorem we obtain that this object is a sofic $\ZZ^2$-subshift from which we extract an SFT extension.

Finally, we extend this construction to $G$ by adding local rules that ensure that if a $\ZZ^2$-coset codes the point $x \in X$ then the coset of $\ZZ^2$ given by the action of $h \in H$ codes $f_h(x)$. This set of rules is coded as a finite amount of forbidden patterns.

Finally, we define the factor code, and show that it satisfies the required properties.

\subsection{A substitution which encodes an action of $\Aut((\ZZ/p\ZZ)^2)$.}
\label{section.substitution}

Let $p \in \NN$. We define a substitution over a two symbol alphabet which generates a sofic $\ZZ^2$-subshift encoding translations of $p^{m+1}\ZZ^2$ for $m \in \NN$. In the proof of the simulation theorem we will only use the case where $p = 3$, but we prefer to proceed here with more generality.


Let $v \in (\ZZ/ p\ZZ)^2 \setminus \{(0,0)\}$ and $\ag = \{ \tilea, \tileb \}$. The $\ZZ^2$-substitution $\texttt{s}_v : \ag \to \ag^{\{0, \dots, p-1\}^2}$ is defined by:

$$\texttt{s}_v(\tilea)_{z} = \begin{cases}
\tileb \mbox{\ \ \  if } z= v   \\ \tilea  \mbox{ \ \ \ \ in the contrary case. }
\end{cases}$$
$$\texttt{s}_v(\tileb)_{z} = \begin{cases}
\tileb \mbox{\ \ \  if } z \in \{(0,0),v\}   \\ \tilea  \mbox{ \ \ \ \ in the contrary case. } \end{cases}$$

As an example, if $p = 3$ and $v = (1,1)$ we get the following:

\begin{center}
\begin{tikzpicture}
\node at (-1.2,0) {$\texttt{s}_v(\tilea) =$};
\begin{scope}[scale = 0.3, shift = {(-1,-1.5)}]
\draw[black] (0,0)grid(+3,+3);
\fill[color = black] (1,1) rectangle ++(1,1);
\end{scope}
\node at (1.2,0) { };
\end{tikzpicture}\begin{tikzpicture}
\node at (-1.2,0) {$\texttt{s}_v(\tileb) =$};
\begin{scope}[scale = 0.3, shift = {(-1,-1.5)}]
\draw[black] (0,0)grid(+3,+3);
\fill[color = black] (0,0) rectangle ++(1,1);
\fill[color = black] (1,1) rectangle ++(1,1);
\end{scope}\end{tikzpicture}
\end{center}

In this example we obtain that the patterns $\texttt{s}_v^3(\tileb)$ and  $\texttt{s}_v^4(\tileb)$ are:

\begin{figure}[h!]
	\centering
	\include{subst}
\end{figure}

To a substitution $\texttt{s}_v$ we associate the subshift ${\texttt{Sub}}_v$ defined as the set of $\ZZ^2$-configurations such that every subpattern appears in some iteration of the substitution $\texttt{s}_v$. $$\texttt{Sub}_v = \{ x \in \{\tilea, \tileb\}^{\ZZ^2} \mid \forall P \sqsubset x, \exists n \in \NN : P \sqsubset \texttt{s}_v^n(\tileb) \}$$

We remark the following properties of these objects:
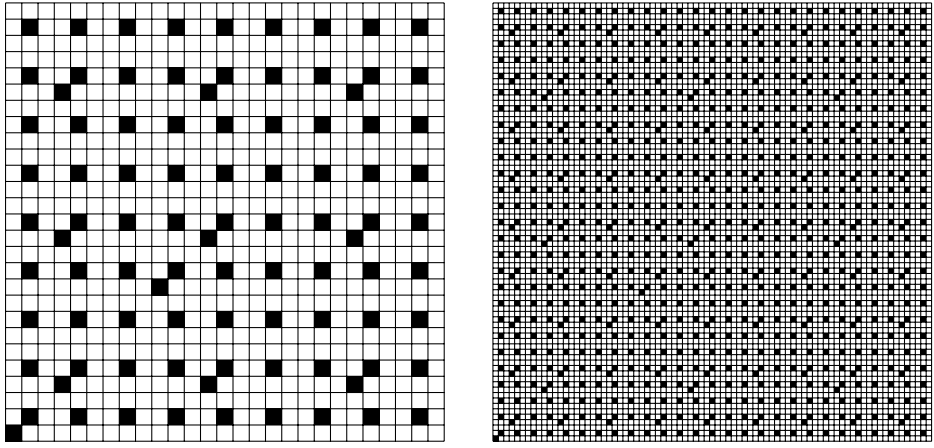
\begin{enumerate}
	\item $s_v$ is a primitive substitution, and thus ${\texttt{Sub}}_v$ is a minimal subshift.
	\item $\forall a \in \{\tilea, \tileb\}$, $n \in \NN$, $\texttt{s}_v^{n+1}(a)_{p^{n}v} = \tileb$. 
	\item By Mozes Theorem~\cite{mozes} $\texttt{Sub}_v$ is a $\ZZ^2$-sofic subshift.  
	\item $\texttt{s}_v$ has unique derivation. This implies by Mozes's results that there is an almost 1-1 SFT extension for $\texttt{Sub}_v$.
	\item Putting together the unique derivation and the second property we obtain the following: $\forall z \in \texttt{Sub}_v$ and $\forall n \in \NN$ there exists a unique $(i_n,j_n) \in P_n := [0,p^{n+1}-1]^2 \cap \ZZ^2$ such that $\forall m \leq n$ then $(i_n,j_n) +p^{m}v+p^{m+1}\ZZ^2$ is composed completely of black squares. We denote each of these sets by $B_m(z)$ -- The lattice of black squares of level $m$. These sets are all disjoint and cover every black square in $\texttt{Sub}_v$ with the possible exception of at most one. We denote this degenerated lattice by $B_{\infty}(z)$, and note that it can either be empty or contain a single position.
	\item Let $z \in \texttt{Sub}_v$ and $\varphi \in \Aut(\ZZ^2)$. We can identify $\varphi$ as an invertible matrix $A_{\varphi} \in GL(\ZZ,2)$ and construct $A_{\widetilde{\varphi}} \in \mathcal{M}( \ZZ/p\ZZ,2)$ by reducing every entry of this matrix modulo $p$. As $\det(A_{\varphi}) \in \{-1,1\}$ then $\det(A_{\widetilde{\varphi}}) \in \{1,p-1\}$. Therefore $A_{\widetilde{\varphi}} \in GL(\ZZ/p\ZZ,2)$ and it is identified as an automorphism $\widetilde{\varphi} \in \Aut((\ZZ/p\ZZ)^2)$.
	
	With this in mind, we obtain the following relation: we have that for $m \leq n$ then $B_m(z) =  (i_n,j_n) +p^{m}v+p^{m+1}\ZZ^2$. Therefore:
	\begin{align*}
	\varphi( B_m(z) ) & = \varphi((i_n,j_n))+ p^{m} \varphi(v)+p^{m+1} \varphi(\ZZ^2) \\
	& = \varphi((i_n,j_n))+ p^{m}A_{\varphi}(v) +p^{m+1} \ZZ^2 \\
	& = \varphi((i_n,j_n))+ p^{m}(A_{\widetilde{\varphi}}  + p( \frac{A_{\varphi}-A_{\widetilde{\varphi}}}{p})  (v) +p^{m+1} \ZZ^2 \\
	& = \varphi((i_n,j_n))+ p^{m}A_{\widetilde{\varphi}}(v) + p^{m+1} (( \frac{A_{\varphi}-A_{\widetilde{\varphi}}}{p})(v)+ \ZZ^2) \\
	& = \varphi((i_n,j_n))+ p^{m}\widetilde{\varphi}(v) +p^{m+1} \ZZ^2 \\
	\end{align*} 
	This means that for fixed $n$ all lattices of size $m \leq n$ are sent to lattices appearing in $\texttt{Sub}_{\widetilde{\varphi}(v)}$. Making $n$ go to infinity and reasoning by compactness we conclude that $\forall z \in \texttt{Sub}_{v}$ there exists $z' \in \texttt{Sub}_{\widetilde{\varphi}(v)}$ such that the image of $(B_m(z))_{m \in \NN}$ under $\varphi$ is $(B_m(z'))_{m \in \NN}$.
\end{enumerate}

We shall use these lattices to encode elements of $\{0,1\}^{\NN}$ belonging to our $H$-flow $(X,f)$. In order to do this, we need to define a subshift which forces to match these lattices to actual values from $X$ and to code the action of $f$. 

\subsection{Encoding configurations in Toeplitz sequences.}

Consider $p \geq 3, q \in \{1,\dots,p-1\}$ and the application $\Psi_q : \{0,1\}^{\NN} \to \{0,1,\$\}^{\ZZ}$ given by:

$$ \Psi_q(x)_j = \begin{cases}
x_n \mbox{\ \ \  if } j = qp^n \mod{p^{n+1}} \\ \$  \mbox{ \ \ \ \ in the contrary case. }
\end{cases} $$

The idea behind this encoding is to match for each $m \in \NN$ the projection of the lattice $B_m(x)$ to the symbol $x_m$. We need to do this for every possible choice of $q$ as the projections of the lattice associated to $v = (1,1)$ are different than the ones for $v = (2,2)$ for example.

Every configuration $x \in \{0,1\}^{\NN}$ is encoded in a Toeplitz sequence $\Psi_q(x)$. We begin this section by studying the structure of $\Psi_q(x)$. 

First notice that $\Psi_q(x)|_{q+p\ZZ} \equiv x_0$ and $\forall q' \in \{1,\dots p-1\}\setminus \{q\}$ we have that $\Psi_q(x)_{q'+p\ZZ} \equiv \$ $. Indeed, as $q'+pk \neq 0 \mod{p}$ thus $q'+pk \neq p^i \mod{p^{i+1}}$. Also, if $i \geq 1$ and $\Psi_q(x)_j = x_i$ then $\Psi_q(x)_{j+q} = x_0$ as $j = p^i \mod{p^{i+1}} \implies j = 0 \mod p$. This means that every $x_0$ is a special coordinate in a string of $p-1$ symbols where every other symbol is $\$ $ and every $x_i$ with $i \geq 1$ is necessarily followed by such string. As $p \geq 3$ the lattice of $x_0$ can be recognized as they are the only symbols which are preceded by $q-1$ symbols $\$$ and followed by $p-q-1$ symbols $\$$ and at least one of these two values is positive.

For $x = (x_i)_{i \in \NN} \in \{0,1\}^{\NN}$ let $\sigma(x) \in \{0,1\}^{\NN}$ be defined by $\sigma(x)_i=x_{i+1}$ (we shall use the same notation as in the case of the group shift action, though in this case it's a one-sided $\NN$-action). We define also for $k \in \{0,\dots,p-1\}$ the transformation $\Omega_{k} : \{0,1,\$\}^{\ZZ} \to \{0,1,\$\}^{\ZZ}$ by $(\Omega_{k}(y))_j = y_{jp+k}$. 

\begin{proposition}\label{propositionOMEGA}
	 Let $x \in \{0,1\}^{\NN}$ and $y \in \overline{\text{Orb}_{\sigma}(\Psi_q(x))}$. There exists a unique $k_0 \in \{0,\dots p-1\}$ such that: $$\Omega_{k_0}(y)  \in  \overline{\Orb_{\sigma}(\Psi_q(\sigma(x)))}.$$
\end{proposition}

\begin{proof}
	
The application $\Omega_{k}$ is clearly continuous in the product topology as fixing $y$ in the interval $\ZZ\cap [-lp,lp-1]$ for $l \geq 1$ necessarily fixes $\Omega_k(y)$ in the interval $\ZZ\cap [-l,l-1]$. 

Let $y \in \overline{\Orb_{\sigma}(\Psi_q(x))}$. As $\Psi_q(x)|_{q+p\ZZ} \equiv x_0$ we can deduce by compactness that there exists $k' \in \{1,\dots,p\}$ such that $y|_{k'+p\ZZ} \equiv x_0$. Using the argument that every $x_0$ is in a string of $p-1$ symbols which repeats recurrently we can choose $k_0 := k'-q \mod p$ which satisfies that $x_{k_0+1,\dots k_0+p-1}$ is this string. Consider a sequence $(\sigma_{z_i}(\Psi_q(x)))_{i \in \NN} \to y$. Without loss of generality we can ask that $z_i \in p\ZZ-k_0$, if not it suffices to eliminate a finite number of terms. We get that 	\begin{align*}
\Omega_{k_0}(\sigma_{k_0+pl}(\Psi_q(x)))&= \Omega_{0}(\sigma_{pl}(\Psi_q(x))) \\
&= \sigma_{l}\Omega_{0}(\Psi_q(x)) \\
&= \sigma_l(\Psi_q(\sigma(x))) \in \Orb(\Psi_q(\sigma(x)))
\end{align*}

As $\Omega_k$ is continuous, we obtain that $\Omega_{k_0}(y) \in \overline{\Orb_{\sigma}(\Psi_q(\sigma(x)))}$. 
\end{proof}

\begin{example*}
	For $p = 3$, $q = 1$ and $x = x_0x_1x_2\dots$ we obtain that:
	
	$$\Psi_q(x)|_{\{0,\dots,30\}} = \$x_0\$x_1x_0\$\$x_0\$x_2x_0\$x_1x_0\$\$x_0\$\$x_0\$x_1x_0\$\$x_0\$x_3x_0\$x_1 $$
	$$\Omega_0(\Psi_q(x))|_{\{0,\dots,10\}} = \$x_1\$x_2x_1\$\$x_1\$x_3x_1 = \Psi_q(\sigma(x))|_{\{0,\dots,10\}} $$
	$$\Omega_0^2(\Psi_q(x))|_{\{0,\dots,3\}} = \$x_2\$x_3 = \Psi_q(\sigma^2(x))|_{\{0,\dots,3\}} $$
\end{example*}

The previous proposition actually shows that $x$ can be decoded from any element of the closure of the orbit of $\Psi_q(x)$ under the shift action. That necessarily implies that the orbits are disjoint.

\begin{proposition}\label{propositionorbitsaredisjoint}
	Let $x,x' \in X$. If $x \neq x'$ then $\overline{\Orb_{\sigma}(\Psi_q(x))} \cap \overline{\Orb_{\sigma}(\Psi_q(x'))} = \emptyset$
\end{proposition}

\begin{proof}
	
	Let $y \in \overline{\Orb_{\sigma}(\Psi_q(x))} \cap \overline{\Orb_{\sigma}(\Psi_q(x'))}$. Using Proposition~\ref{propositionOMEGA} we can find $k_0$ such that $\Omega_{k_0}(y)  \in  \overline{\Orb_{\sigma}(\Psi_q(\sigma(x)))}$. Moreover, we get that $x_0 = x'_0 = y_{k_0}$. Iterating this procedure we obtain that $\forall i \in \NN$ then $x_i = x'_i$ and thus $x = x'$.
\end{proof}


Before continuing, let's draw the attention to the structure of the subshift $\overline{\Orb_{\sigma}(\Psi_q(x))}$. Every element here encodes the structure of $x$ by repeating its $n$-th coordinate in gaps of size $p^{n+1}$. Therefore, every non $\$$ element appears periodically with at most one exception -- a position obtained by compactness -- which we denote by $x_{\infty}$. This point may take any value if both $0$ and $1$ appear infinitely often in $x$ but is restricted if $x$ is eventually constant. This point is analogous to the lattice $B_{\infty}(z)$ appearing in the substitution we defined before.

Let $(X,f)$ be an $H$-flow and $p \geq 3$. We use the encoding $\Psi_q$ defined above to construct a $\ZZ$-subshift $\texttt{Top}(X,f)$ which encodes the points of $x$ and the action of $f$ around a unit ball in $H$. Formally, let $S \subset H$ be a finite set such that $1_H \in S$ and $\langle S \rangle = H$. 

$\texttt{Top}(X,f) \subset (\{0,1,\$\}^{(p-1)|S|})^{\ZZ}$ is the $\ZZ$-subshift given by:

$$ \texttt{Top}(X,f) := \bigcup_{x \in X}\left(\overline{\Orb_{\sigma}\left(\Psi_q(f_s(x))_{(q,s) \in \{1,\dots, p-1\} \times S}\right)}\right) $$

Elements of $\texttt{Top}(X,f)$ are $(p-1)|S|$-tuples which encode elements of the $\ZZ$-orbit of each $\Psi_q(f_{s}(x))$. The idea behind this construction is to let each $q$-row code an element $x \in X$ and its image $f_s(x)$ for each $s \in S$. Given $y  \in \texttt{Top}(X,f)$ we denote the projection to the $q,s$-th layer by $\texttt{Layer}_{q,s}(y) \in \{0,1,\$ \}^{\ZZ}$. We need to do this for every possible $q$ just for technical reasons, as we'll need to match every possible lattice in the substitution defined above. For all practical purposes, just one coordinate $q$ carries all the information we need to code.


\begin{proposition}\label{Xtildeiseffective}
	If $(X,f)$ is an effectively closed $H$-flow then $\texttt{Top}(X,f)$ is an effectively closed $\ZZ$-subshift.
\end{proposition}

\begin{proof}
$\texttt{Top}(X,f)$ is clearly shift invariant. It is closed as $X$ is closed and thus a diagonal argument allows to extract convergent subsequences. A set of forbidden patterns defining $\texttt{Top}(X,f)$ is the following. We consider for $n \in \NN$ all words of length $p^{n+1}$ over the alphabet $\{\$,0,1\}^{|S|(p-1)}$ which do not appear in any configuration of $\texttt{Top}(X,f)$. As this is an increasing sequence of forbidden patterns it is enough to define $\texttt{Top}(X,f)$.
	
	This set of forbidden words is recursively enumerable. The following algorithm accepts a set of forbidden patterns defining $\texttt{Top}(X,f)$. Let the input be a word of length $p^{n}$ for $n \in \NN$. The structure of $\texttt{Top}(X,f)$ makes it possible to recognize algorithmically all gaps in every layer (formally the algorithm checks that each substring of $p$ contiguous symbols is a cyclic permutation of $a\$^{q-1}b\$^{p-q-1}$ for some $a \in \{0,1,\$\}$ and $b \in \{0,1\}$). Then if this stage is passed, it computes $k_0$ from Proposition~\ref{propositionOMEGA} for each layer, checks that $b$ is the same symbol throughout the word. Finally it checks that $k_0$ is the same in every layer (thus the layers are aligned). Then it applies $\Omega_{k_0}$ to this string obtaining a word of length $p^{n-1}$. The algorithm is repeated until reaching a word of length $0$. If at any stage a check fails, the word is accepted as forbidden.
	
	The previous stage recognizes all words that haven't got the correct structure. After that stage ends, we can use the same algorithm and the function $\Omega_k$ to decode $n$ coordinates $x_0x_1\dots x_{n-1}$ for each pair $(q,s)$ and check for every $s$ that the word is the same independently of $q$. If this stage is passed we end up with $|S|$ words which depend only on $s$ and we denote them by $(w_s)_{s \in S}$. Here we run two recognition algorithms in parallel. One searches for a cylinder $[w_s] \not\subset X$ and the other searches if $[w_{1_H}] \not\subset f^{-1}_s([w_s])$. If any of these two searches succeed at a certain step then the algorithm returns that the pattern is forbidden. These two last algorithms do exists as $(X,f)$ is an effectively closed $H$-flow. \end{proof}
	
The subshift $\texttt{Top}(X,f)$ is the ingredient of the proof which allows us to simulate points $x \in X$ and their images under the generators of $H$ in a sofic $\ZZ^2$-subshift which contains this information. The next step is to put one of these configurations in each $\ZZ^2$-coset of $\ZZ^2 \rtimes_{\varphi} H$ and force by local rules that the shift action by $(0,h)$ yields the $\ZZ^2$-coset where the point $f_h(x)$ is codified. The obvious obstruction to this idea is the fact that the action under $(0,h)$ in a semidirect product disturbs the adjacency relations in a coset if the automorphism $\varphi_h$ isn't trivial. The way to go around this obstruction is to use the lattices given by the layer $\texttt{Sub}_v$ which are invariant under automorphisms. We specify how these two elements go together in the next subsection.

\subsection{Proof of Theorem~\ref{simulationTHEOREM}}

Denote $\varphi : H \to \Aut(\ZZ^2)$ a group homomorphism such that $G = \ZZ^2 \rtimes_{\varphi} H$ is given by: $$(n_1,h_1)\cdot(n_2,h_2) = (n_1 + \varphi_{h_1}(n_2),h_1h_2)$$ 

To make notations shorter, we write $\vec{0} = (0,0) \in \ZZ^2$ throughout the whole proof. Let $S$ be a finite set of generators of $H$ where $1_H \in S$ , $|S| = d$ and let's fix the parameter $p = 3$ which is used to construct $\texttt{Top}(X,f)$ (which contains thus $2d$ layers) and the substitutions $\texttt{Sub}_v$ for $v \in (\ZZ/3\ZZ)^2 \setminus \{\vec{0}\}$. Consider the following two $\ZZ^2$-subshifts.
		
 $$\texttt{Top}(X,f)^{H} \subseteq (\{0,1,\$\}^{2d})^{\ZZ^2}$$ 
  $$\texttt{Top}(X,f)^{V} \subseteq (\{0,1,\$\}^{2d})^{\ZZ^2}$$ 
 
 Where $x \in \texttt{Top}(X,f)^{H}$ is the subshift whose projection to $(\ZZ,0)$ belongs to $\texttt{Top}(X,f)$ and any vertical strip is constant. Analogously $x \in \texttt{Top}(X,f)^{V}$ is the subshift whose projection to $(0,\ZZ)$ belongs to $\texttt{Top}(X,f)$ and any horizontal strip is constant. Formally: $x \in \texttt{Top}(X,f)^{H}$ if $\forall i,j \in \ZZ$ then $x_{i,j}=x_{i,j+1}$ and $(x_{(i,0)})_{i \in \ZZ} \in \texttt{Top}(X,f)$. An analogous definition can be given for $\texttt{Top}(X,f)^{V}$. Proposition~\ref{Xtildeiseffective} says that $\texttt{Top}(X,f)$ is an effective $\ZZ$-subshift and therefore $\texttt{Top}(X,f)^{H}$ and $\texttt{Top}(X,f)^{V}$ are sofic $\ZZ^2$-subshifts by the simulation theorem proven in \cite{AubrunSablik2010, DurandRomashchenkoShen2010}. Next we are going to put these subshifts together with the substitution layers to create a rich structure in each $\ZZ^2$-coset.

	Let $\Pi(X,f) \subset \texttt{Top}(X,f)^{H} \times \texttt{Top}(X,f)^{V} \times \bigotimes_{v \in (\ZZ/3\ZZ)^2 \setminus \{\vec{0}\} }\texttt{Sub}_v$ be the $\ZZ^2$-subshift defined by forbidding the following symbols in the product alphabet. In order to describe the forbidden symbols correctly, we introduce the following notation: For $y \in \Pi(X,f)$ we denote by $\texttt{Layer}^H_{q,s}(y)$ and $\texttt{Layer}^V_{q,s}(y)$ the projections to the first and second layers in the $(q,s)$ coordinate respectively and for $v \in (\ZZ/3\ZZ)^2 \setminus \{\vec{0}\}$ we denote by $\texttt{Sub}_v(y)$ the projection to the corresponding substitutive layer.

	\begin{enumerate}
		\item $\forall (i,j) \in \ZZ^2$ and $(a,b) \in (\ZZ/3\ZZ)^2 \setminus \{\vec{0}\}$ the following is satisfied. If $a \neq 0$ then $(\texttt{Sub}_{(a,b)}(y))_{(i,j)} = \tileb$ if and only if $(\texttt{Layer}^H_{a,1_H}(y))_{(i,j)} \in \{0,1\}$. Analogously, if $b \neq 0$ then then $(\texttt{Sub}_{(a,b)}(y))_{(i,j)} = \tileb$ if and only if $(\texttt{Layer}^V_{b,1_H}(y))_{(i,j)} \in \{0,1\}$.
		\item If $(\texttt{Sub}_{(1,1)}(y))_{(i,j)} = \tileb$ then $\forall s \in S$  $(\texttt{Layer}^H_{1,s}(y))_{(i,j)} = (\texttt{Layer}^V_{1,s}(y))_{(i,j)}$.
	\end{enumerate}
	
	The $\ZZ^2$-subshift $\Pi(X,f)$ is sofic. Indeed, all the component are sofic subshifts and the added rules are local (we just forbid symbols in the product alphabet). In what follows we use the following notation: for a configuration $x \in \ag^G$, $A \subset G$ and $a \in \ag$ we write $x|_{A} \equiv a$ if $\forall g \in A$ then $x_g = a$.  Recall that we denote by $B_m(z)$ the $m$-th lattice of black squares in a configuration $z$ in a substitutive layer. 
	
	\begin{claim}\label{propositionlatticescarryinformation}
			Let $y \in \Pi(X,f)$, $(a,b) \in  (\ZZ/3\ZZ)^2 \setminus \{\vec{0}\}$ and $z = \texttt{Sub}_{(a,b)}(y)$. Suppose that $\texttt{Layer}^H_{a,1_H}(y)$ is given by $x \in X$. Then:
			\begin{itemize}
				\item  If $a \neq 0$ then $\forall m \in \NN, \forall s \in S$: $\texttt{Layer}^H_{a,s}(y)|_{B_m(z)} \equiv f_s(x)_m$
				\item  If $b \neq 0$ then $\forall m \in \NN, \forall s \in S$: $\texttt{Layer}^V_{b,s}(y)|_{B_m(z)} \equiv f_s(x)_m$
				\item  The configurations in the layers $\texttt{Top}(X,f)^{H}$ and $\texttt{Top}(X,f)^{V}$ are defined by the same $x \in X$.
			\end{itemize}
	\end{claim}
	
	\begin{proof}
		Let $a \neq 0$. It suffices to show this property for $s = 1_H$ as the definition of $\texttt{Top}(X,f)$ forces the configurations to be aligned. The lattice $B_0(z)$ has the form $(i_0,j_0) + (a,b)+3\ZZ^2$, therefore its projection in the horizontal coordinate is $k_0+3\ZZ$ for $k_0 = i_0+a \mod{3}$. Using the structure of $\Psi_a(x)$ there are three possibilities for $3$-lattices: One contains uniformly the symbol $x_0$, another contains only the symbol $\$$ and the third one contains $\Psi_a(\sigma(x))$ by proposition~\ref{propositionOMEGA}. The first rule of $\Pi(X,f)$ rules out the second and third possibility because there would be $\$$'s matched with $\tileb$. Therefore $\texttt{Layer}^H_{a,1_H}|_{B_0(z)} \equiv x_0$. Inductively, let $B_m(z) = (i_m,j_m)+ (a,b)3^m+3^{m+1}\ZZ^2$ and suppose $\forall m' < m$ $\texttt{Layer}^H_{a,1_H}|_{B_{m'}(z)} \equiv x_{m'}$. Note that for $m'$ the projection to the horizontal layer is $k_{m'}+3^{m'+1}\ZZ$ for  $k_{m'} := i_{m}+a3^{m'} \mod{3^{m'+1}}$. Using iteratively the previous argument and applying the function $\Omega_{k_{m'}}$ defined in~\ref{propositionOMEGA} we end up with three possibilities for $3^m$-lattices (that is, the value of $k_{m'}$), and again the first rule of $\Pi(X,f)$ rules out two of them, yielding $\texttt{Layer}^H_{a,1_H}|_{B_{m}(z)} \equiv x_{m}$.
		
		Suppose the configuration in $\texttt{Top}(X,f)^{V}$ is given by $x' \in X$. For $b$ the proof is analogous and we get that $b \neq 0$ implies that $\forall m \in \NN, \forall s \in S$: $\texttt{Layer}^V_{b,s}|_{B_m(z)} \equiv f_s(x')_m$. 
		
		Now set $(a,b) = (1,1)$. The second rule of $\Pi(X,f)$ implies that $\forall s \in S, m \in \NN$ then $(\texttt{Layer}^H_{1,s}(y))|_{B_m(z)}= (\texttt{Layer}^V_{1,s}(y))|_{B_m(z)}$. Using the previous two properties we conclude that $\forall s \in S, m \in \NN$ we have $f_s(x)_m = f_s(x')_m$. Using $s = 1_H$ yields $x = x'$ hence proving the second and third statement.
\end{proof}

From Claim~\ref{propositionlatticescarryinformation} we obtain that each configuration $y \in \Pi(X,f)$ contains the information of a single $x \in X$. We can thus define properly the decoding function $\Upsilon: \Pi(X,f) \to X$ such that $\Upsilon(y) = x$ if and only if $\forall m \in \NN$: $\texttt{Layer}^H_{1,1_H}(y)|_{B_m(\texttt{Sub}_{(1,1)}(y))} \equiv x_m$.

Consider the set of forbidden patterns $\FF$ defining $\Pi(X,f)$. Each of these patterns has a finite support $F \subset \ZZ^2$. We extend those patterns to patterns in $G = \ZZ^2 \rtimes_{\varphi} H$ by associating $d \in F \to (d,1_H) \in G$. Therefore every pattern $P \in \FF$ with support $F \subset \ZZ^2$ is embedded into a pattern $\widetilde{p}$ with support $(F,1_H) \subset G$. We consider the set $\widetilde{\FF} = \{\widetilde{p} \mid p \in \FF\}$ and we define $\texttt{Final}(X,f)$ as the subshift over the same alphabet as $\Pi(X,f)$ defined by the set of forbidden patterns $\widetilde{\FF} \cup \mathcal{G}$ where $\mathcal{G}$ is defined as follows:

For each $s \in S$ consider $\varphi_{s^{-1}}$ the automorphism associated to $s^{-1}$ and $(a,b) = \widetilde{\varphi}_{s^{-1}}(1,1)$. We put in $\mathcal{G}$ all the patterns $P$ with support $\{(\vec{0},1_H), (\vec{0},s^{-1})\}$ which satisfy that $\texttt{Sub}_{(a,b)}(P_{(\vec{0},1_H)}) = \tileb$ but either:
\begin{itemize}
	\item  $\texttt{Sub}_{(1,1)}(P_{(\vec{0},s^{-1})}) \neq \tileb$ or
	\item $\texttt{Sub}_{(1,1)}(P_{(\vec{0},s^{-1})}) = \tileb$ and
	\begin{itemize}
		\item If $a \neq 0$ then $\texttt{Layer}^H_{a,s}(P_{(\vec{0},1_H)}) \neq \texttt{Layer}^H_{1,1_H}(P_{(\vec{0},s^{-1})}) $ or
		\item If $b \neq 0$ then $\texttt{Layer}^V_{b,s}(P_{(\vec{0},1_H)}) \neq \texttt{Layer}^V_{1,1_H}(P_{(\vec{0},s^{-1})}) $.
	\end{itemize}
\end{itemize}

In simpler words: we force that every $\tileb$ in layer $\texttt{Sub}_{(a,b)}$ of the $(\ZZ^2,1_H)$-coset must be matched with a $\tileb$ in $\texttt{Sub}_{(1,1)}$ in the $(\ZZ^2,s^{-1})$-coset and that if $a \neq 0$ then the symbol in $(\vec{0},1_H)$ in $\texttt{Layer}^H_{a,s}$ is the same as the symbol in $(\vec{0},s^{-1})$ in $\texttt{Layer}^H_{1,1_H}$. If $b \neq 0$ we impose that the symbol in $(\vec{0},1_H)$ in $\texttt{Layer}^V_{b,s}$ is the same as the symbol in $(\vec{0},s^{-1})$ in $\texttt{Layer}^V_{1,1_H}$.

Before continuing let's translate $\widetilde{\FF} \cup \mathcal{G}$ into properties of $\texttt{Final}(X,f)$. In order to do that properly, for $y \in \texttt{Final}(X,f)$ we denote by $\pi(y)$ the $\ZZ^2$-configuration such that $\forall (i,j) \in \ZZ^2$ $\pi(y)_{(i,j)} = y_{((i,j),1_H)}$.

\begin{claim}\label{finalsubshiftproperties}
	$\texttt{Final}(X,f)$ satisfies the following properties:
	\begin{itemize}
		\item $\texttt{Final}(X,f)$ is a sofic $G$-subshift.
		\item Let $y \in \texttt{Final}(X,f)$. Then $\pi(y) \in \Pi(X,f)$.
		\item If $\Upsilon(\pi(y)) = x$ then $\forall h \in H$, $\Upsilon(  \pi( \sigma_{(\vec{0},h)}(y) ) ) = f_h(x)$.
	\end{itemize}
\end{claim}

\begin{proof}
	As $\Pi(X,f)$ is sofic, it admits an SFT extension $\phi : \widehat{\Pi}(X,f) \twoheadrightarrow \Pi(X,f)$. By embedding as above a finite list of forbidden patterns defining $\widehat{\Pi}(X,f)$ into $G$ we obtain a $G$-SFT extension of $X_{\widetilde{\FF}}$. Adding to this list of forbidden patterns the pullback of the finite list of forbidden patterns $\mathcal{G}$ under the local code $\Phi$ defining $\phi$ we obtain an SFT extension $\widehat{\texttt{Final}}(X,f)$ of $\texttt{Final}(X,f)$.
	
	The second property comes directly from the definition of $\texttt{Final}(X,f)$ as it contains an embedding of every forbidden pattern defining $\Pi(X,f)$. Note that as $G$ is not necessarily abelian, it may happen that $y|_{(\ZZ^2,h)}$ seen as a $\ZZ^2$-configuration does not belong to $\Pi(X,f)$ for some $h \in H$, but $\pi(\sigma_{(\vec{0},h^{-1})}(y))$ always does. 
	
	Let's prove the third property: We claim that it suffices to prove the property for $s \in S$. Indeed, given $h\in H$, as $H = \langle S \rangle$ there exists a minimal length word representing $h$. If $h = 1_H$ the result is immediate. If not, then $h = sh'$ for some $h' \in H$ having a shorter word representation. Suppose this third property holds for all words of strictly smaller length and define $y' = \sigma_{(\vec{0},h')}(y)$ we have that $\Upsilon(\pi(y')) = f_{h'}(x) = x'$, so:
	$$\Upsilon(\pi(\sigma_{(\vec{0},h)}(y))) =  \Upsilon(\pi(\sigma_{(\vec{0},s)}(y'))) = f_s(x') = f_{s}(f_{h'}(x)) = f_h(x).$$
	It suffices therefore to prove the property for $s \in S$. Let's denote $y' = \sigma_{(\vec{0},s)}(y)$ and let $\Upsilon(\pi(y)) = x$ and  $\Upsilon(\pi(y')) = x'$. We want to prove that $x' = f_s(x)$. Let $\widetilde{\varphi}_{s^{-1}}(1,1) = (a,b)$ and suppose that $a \neq 0$ (if $a = 0$ then $b \neq 0$ and the argument is analogous). Let $m \in \NN$, using Claim~\ref{propositionlatticescarryinformation} we obtain $$\texttt{Layer}^H_{1,1_H}(y')|_{(B_m(\texttt{Sub}_{(1,1)}(y')),1_H)} \equiv x'_m$$
	$$\texttt{Layer}^H_{a,s}(y)|_{(B_m(\texttt{Sub}_{(a,b)}(y)),1_H)} \equiv f_s(x)_m.$$
	
	Using the forbidden patterns $\mathcal{G}$ results in $$\texttt{Sub}_{(1,1)}(y)|_{(B_m(\texttt{Sub}_{(a,b)}(y)),s^{-1})} \equiv \tileb$$ $$\texttt{Layer}^H_{1,1_H}(y)|_{(B_m(\texttt{Sub}_{(a,b)}(y)),s^{-1})} \equiv f_s(x)_m.$$
	
	Finally, developing the action on $y'$ yields \begin{align*}
	 y'|_{(B_m(\texttt{Sub}_{(1,1)}(y')),1_H)} &=  \sigma_{(\vec{0},s)}(y)|_{(B_m(\texttt{Sub}_{(1,1)}(y')),1_H)} \\
	 & = y|_{(\vec{0},s^{-1})  (B_m(\texttt{Sub}_{(1,1)}(y')),1_H)} \\
	 & = y|_{(\varphi_{s^{-1}}(B_m(\texttt{Sub}_{(1,1)}(y'))),s^{-1})}.
	 \end{align*}
	
	Using the results from Section~\ref{section.substitution} we obtain that $B_m(\texttt{Sub}_{(1,1)}(y')$ is of the form $(i_m,j_m)+(1,1)3^{m}+3^{m+1}\ZZ^2$ and thus $\varphi_{s^{-1}}(B_m(\texttt{Sub}_{(1,1)}(y'))$ is  $\varphi_{s^{-1}}(i_m,j_m)+(a,b)3^{m}+3^{m+1}\ZZ^2$. Which is $B_m(z)$ for some $z \in \texttt{Sub}_{(a,b)}$. As we have $\forall m \in \NN$ that $\texttt{Sub}_{(1,1)}(y')|_{(B_m(\texttt{Sub}_{(1,1)}(y')),1_H)} \equiv \tileb$ and $\texttt{Sub}_{(1,1)}(y)|_{(B_m(\texttt{Sub}_{(a,b)}(y)),s^{-1})} \equiv \tileb$ we conclude that $ \varphi_{s^{-1}}(B_m(\texttt{Sub}_{(1,1)}(y')) = B_m(\texttt{Sub}_{(a,b)}(y))$. Therefore, $$\texttt{Layer}^H_{1,1_H}(y')|_{(B_m(\texttt{Sub}_{(1,1)}(y')),1_H)} = \texttt{Layer}^H_{1,1_H}(y)|_{(B_m(\texttt{Sub}_{(a,b)}(y)),s^{-1})}.$$
	Which yields $x'_m = f_s(x)_m$. As $m \in \NN$ is arbitrary $x' = f_s(x)$.\end{proof}
	
	Finally we are ready to finish the proof. Consider again the SFT extension $\widehat{\texttt{Final}}(X,f)$ of $\texttt{Final}(X,f)$, the factor map $\phi: \widehat{\texttt{Final}}(X,f) \twoheadrightarrow \texttt{Final}(X,f)$ and the subaction $(\widehat{\texttt{Final}}(X,f), \sigma^H)$.
	
	\begin{proposition}\label{proposition.final}
		$\Upsilon \circ \pi \circ \phi$ is a factor map from $(\widehat{\texttt{Final}}(X,f),\sigma^H)$ to $(X,f)$.
	\end{proposition}
	
	\begin{proof}
		As  $\phi: \widehat{\texttt{Final}}(X,f) \twoheadrightarrow \texttt{Final}(X,f)$ it suffices to show that $\Upsilon \circ \pi$ is a factor map from $(\texttt{Final}(X,f),\sigma^H)$ to $(X,f)$. Let $y \in \texttt{Final}(X,f)$. Following Claim~\ref{finalsubshiftproperties} we have $\pi(y) \in \Pi(X,f)$ and thus $\Upsilon(\pi(y))  \in X$. Moreover, setting $\Upsilon(\pi(y)) = x$ yields $\forall h \in H$ that $\Upsilon(\sigma_{(\vec{0},h)}(y)) = f_h(x)$. This implies
		 $$\forall h \in H: (\Upsilon \circ \pi ) \circ \sigma_{(\vec{0},h)} = f_h \circ (\Upsilon \circ \pi).$$
		 
		 Also, both $\Upsilon$ and $\pi$ are clearly continuous, therefore, it only remains to show that $\Upsilon \circ \pi$ is surjective. Let $x \in X$, we construct a configuration $y^* \in \texttt{Final}(X,f)$ such that $\Upsilon(\pi(y^*)) = x$.
		 
		  In order to do this, we begin by constructing a sequence of configurations $(y^h)_{h \in H}$ which belong to $\Pi(X,f)$. For $(a,b) \in  (\ZZ/3\ZZ)^2 \setminus \{\vec{0}\}$ let $z_{(a,b)} \in \texttt{Sub}_{(a,b)}$ defined by $(i_n,j_m) = 0$ for all $m \in \NN$. Said otherwise, $B_m(z_{(a,b)}) = (a,b)3^m + 3^{m+1}\ZZ^2$ for $m \in \NN$ and $B_{\infty}(z_{(a,b)})= \emptyset$. We define $y^h \in \Pi(X,f)$ by $\texttt{Sub}_{(a,b)}(y^h) = z_{(a,b)}$ and $\forall (i,j) \in \ZZ^2$, $s \in S$, $a,b \in \{1,2\}$ then $\texttt{Layer}^H_{a,s}(y^h)_{(i,j)} = \Psi_a(f_s(f_h(x)))_{i}$ and $\texttt{Layer}^V_{b,s}(y^h)_{(i,j)} = \Psi_b(f_s(f_h(x)))_{j}$. It can easily be verified that for each $h \in H$ the configuration $y^h \in \Pi(X,f)$.
		  
		  Finally, we define $y^*$ as follows: $$ (y^*)_{( (i,j), h)} = (y^{h^{-1}})_{\varphi_{h^{-1}}(i,j)}. $$
		As $\varphi_{1_H}(i,j) = (i,j)$ then $\pi(y^*) = y^{1_H}$ and thus $\Upsilon(\pi(y^*)) = f_{1_H}(x) = x$. It suffices to show that $y^* \in \texttt{Final}(X,f)$. This comes down to showing that no patterns in $\FF$ or $\mathcal{G}$ appear in $y^*$. Suppose a pattern $P \in \FF$ appears at position $g = ((i,j),h)$, that is $y^* \in [P]_g \iff \sigma_{g^{-1}}(y^*) \in [P]_{1_G}$. As $P$ has a support contained in $(\ZZ^2, 1_H)$ then $\pi(\sigma_{g^{-1}}(y^*)) \notin \Pi(X,f)$. But \begin{align*}
		\sigma_{g^{-1}}(y^*)_{((i',j'),1_H)} & = (y^*)_{g((i',j'),1_H)} \\
		& = (y^*)_{((i,j)+\varphi_h(i',j'),h)} \\
		& = (y^{h^{-1}})_{(i',j')+\varphi_{h^{-1}}(i,j)} \\
		& = (\sigma_{-\varphi_{h^{-1}}(i,j)}(y^{h^{-1}}))_{(i',j')}.
		\end{align*} Therefore $\pi(\sigma_{g^{-1}}(y^*)) = \sigma_{-\varphi_{h^{-1}}(i,j)}(y^{h^{-1}}) \in \Pi(X,f)$ which is a contradiction. Hence $y^*$ does not contain any pattern from $\FF$. It remains to show it contains no patterns in $\mathcal{G}$. Recall that patterns $P \in \mathcal{G}$ have support $\{ (\vec{0},1_H),(\vec{0},s^{-1})\}$ for $s \in S$. Let $g = ((i,j),h)$ such that $\sigma_{g^{-1}}(y^*) \in [P]_{1_G}$. Then $\sigma_{g^{-1}}(y^*)_{(\vec{0},1_H)} = \sigma_{-\varphi_{h^{-1}}(i,j)}(y^{h^{-1}})_{\vec{0}}$ and \begin{align*}
		\sigma_{g^{-1}}(y^*)_{(\vec{0},s^{-1})}  & = (y^*)_{((i,j),h)(\vec{0},s^{-1})} \\ & = (y^*)_{((i,j),hs^{-1})} \\& = (y^{sh^{-1}})_{\varphi_{(sh^{-1})}(i,j)} \\& = (\sigma_{-(\varphi_{(sh^{-1})}(i,j))}(y^{sh^{-1}}))_{\vec{0}}.\end{align*} 
		Let $m \in \NN$ and denote $(a,b) = \widetilde{\varphi}_{s^{-1}}(1,1)$. By definition $B_m(\texttt{Sub}_{(a,b)}(y^h) = (a,b)3^m+3^{m+1}\ZZ^2$ therefore,
		
		$$B_m(\texttt{Sub}_{(a,b)}(\sigma_{-\varphi_{h^{-1}}(i,j)}(y^h)) = (a,b)3^m+\varphi_{h^{-1}}(i,j)+3^{m+1}\ZZ^2$$
		
		In the other hand, $$B_m(\texttt{Sub}_{(1,1)}(\sigma_{-(\varphi_{(sh^{-1})}(i,j))}(y^{sh^{-1}})) = (1,1)3^m+\varphi_{(sh^{-1})}(i,j)+3^{m+1}\ZZ^2.$$
		
		So, if $\texttt{Sub}_{(a,b)}(\sigma_{g^{-1}}(y^*))_{(\vec{0},1_H)} = \tileb$ then $\vec{0} \in (a,b)3^m+\varphi_{h^{-1}}(i,j)+3^{m+1}\ZZ^2$ for some $m \in \NN$. Applying $\varphi_s$ at both sides we obtain:
		
		\begin{align*}
		\varphi_s(\vec{0})= \vec{0} & \in \varphi_s(a,b)3^m + \varphi_{(sh^{-1})}(i,j) +3^{m+1}\ZZ^2 \\
		& = \widetilde{\varphi}_s(a,b)3^m + \varphi_{(sh^{-1})}(i,j) +3^{m+1}\ZZ^2 \\
		& = (1,1)3^m + \varphi_{(sh^{-1})}(i,j) +3^{m+1}\ZZ^2 \\
		& = B_m(\texttt{Sub}_{(1,1)}(\sigma_{-(\varphi_{(sh^{-1})}(i,j))}(y^{sh^{-1}})).
		\end{align*}
		
		Implying that $\texttt{Sub}_{(1,1)}(\sigma_{g^{-1}}(y^*))_{(\vec{0},s^{-1})} = \tileb$. Moreover, if either $a$ is non-zero (the $b \neq 0$ case is analogous), then, using the previous computation we get: $$\texttt{Layer}^H_{a,s}(\sigma_{g^{-1}}(y^*))_{(\vec{0},1_H)} = f_s(f_{h^{-1}}(x))_m$$ $$\texttt{Layer}^H_{1,1_H}(\sigma_{g^{-1}}(y^*))_{(\vec{0},1_H)} = f_{sh^{-1}}(x)_m.$$
		
		So no patterns from $\mathcal{G}$ appear, yielding $y^* \in \texttt{Final}(X,f)$.\end{proof}
	
	Proposition~\ref{proposition.final} concludes the proof of Theorem~\ref{simulationTHEOREM}.
	
\section{Consequences and remarks}\label{section.consequences}

In this last section we explore some consequences of our simulation theorem. The first one is the case of expansive actions. Here we show that as long as the group is recursively presented, the action can be presented in a convenient way that allows us to replace the subaction by the projective subdynamics. The second is an application of this theorem to produce non-empty strongly aperiodic subshifts in a class of groups where this fact whas previously unknown, answering a question posed by Sahin in a symbolic dynamics workshop held in Chile in december 2014. We also extend a Theorem of Jeandel~\cite{Jeandel2015} to the existence of effectively closed strongly aperiodic flows in general.

We close this section by remarking that the technique used to prove Theorem~\ref{simulationTHEOREM} is valid in an even larger class (namely, simulation in $\ZZ^d \rtimes G$) and with a discussion on the size of the extension. Indeed, in Hochman's article~\cite{Hochman2009b} the subaction is shown to be an almost trivial isometric extension. We dedicate the last part of this section to informally discuss the size of the factor in our construction and how a similar result could be obtained.

\subsection{The simulation theorem for expansive effective flows}

Before presenting the simulation theorem for expansive actions, we must define with more detail effectively closed subshifts in groups. A longer survey of these concepts can be found in~\cite{ABS2014}. Given a group $G$ generated by $S$ and a finite alphabet $\ag$ a \define{pattern coding} $c$ is a finite set of tuples $c=(w_i,a_i)_{i \in I}$ where $w_i \in S^{*}$ and $a_i \in \ag$. A set of pattern codings $\CC$ is said to be recursively enumerable if there is a Turing machine which takes as input a pattern coding $c$ and accepts it if and only if $c \in \CC$. A subshift $X \subset \ag^G$ is \define{effectively closed} if there is a recursively enumerable set of pattern codings $\CC$ such that:$$ X = X_{\CC} := \bigcap_{g \in G, c \in \CC} \left( \ag^G \setminus \bigcap_{(w,a) \in c}[a]_{gw} \right).$$

With this formal concept in hand, we show the following lemma:

\begin{lemma}\label{lemma_symbolicfactoriseffective}
	For every finitely generated group, any $G$-subshift which is the factor of an effectively closed $G$-flow is itself effectively closed.
\end{lemma}

\begin{proof} 
		Let $G$ be generated by the finite set $S \subset G$, $(X,f)$ an effectively closed $G$-flow over a Cantor set, $(Y,\sigma)$ a $G$-subshift and $\phi: (X,f) \twoheadrightarrow (Y,\sigma)$ a factor. 
		
		Recall that $X \subset \{0,1\}^{\NN}$ and $Y \subset \ag^G$ for some finite $\ag$. As both $X$ and $Y$ are compact, $\phi$ is uniformly continuous. Therefore for each $a \in \ag$ then $\phi^{-1}([a]) = W_a$ where $W_a$ is a clopen set depending on a finite number of coordinates. For any pattern coding $c $: $$\phi^{-1}\left( \bigcap_{(w,a) \in c}[a]_{w} \right) = \bigcap_{(w,a) \in c}\phi^{-1}(\sigma_{w^{-1}}([a]) ) = \bigcap_{(w,a) \in c} f_{w^{-1}}( \phi^{-1}([a]) )$$
	
	Therefore,
	$$Y \cap \bigcap_{(w,a) \in c}[a]_{w} = \emptyset \implies X \cap \bigcap_{(w,a) \in c}f_{w^{-1}}( W_{a} ) = \emptyset.$$
	
	As $(X,f)$ if effectively closed, there is a Turing machine which can approximate the set $\bigcap_{(w,a) \in c}f_{w^{-1}}( W_{a} )$ as each $W_a$ is just a finite union of a finite intersection of cylinders and $w^{-1} \in S^*$. Also, for each partial approximation we can use the enumeration cylinders defining the complement of $X$ to recognize if the intersection is empty, using these tools we can construct a Turing machine recognizing a maximal set of forbidden pattern codings defining $Y$.\end{proof}

\begin{theorem}\label{theorem.projective}
	Let $H$ be a finitely generated and recursively presented group.
	\begin{enumerate}
		\item Let $(X,f)$ be an effectively closed expansive $H$-flow over a Cantor set. Then there exists a $(\ZZ^2 \rtimes H)$-sofic subshift $Y$ such that its $H$-projective subdynamics $\pi_H(Y)$ is conjugated to $(X,f)$.
		\item Let $Z$ be an effectively closed $H$-subshift. Then there exists a sofic $(\ZZ^2 \rtimes H)$-subshift $Y$ such that its $H$-projective subdynamics $\pi_H(Y)$ is $Z$.
	\end{enumerate}
\end{theorem}

\begin{proof}
	Let $S \subset H$ be a finite set such that $\langle S \rangle = H$. Consider a recursive bijection $\varphi : \NN \to S^*$ where $S^*$ is the set of all words on $S$ and $\varphi(0) = \epsilon$ the empty word. As $H$ is recursively presented, then its word problem $\texttt{WP}(H) = \{w \in S^* \mid w = 1_H \}$ is recursively enumerable and thus there is a Turing machine $T$ which accepts all pairs $(n,n') \in \NN^2$ such that $\varphi(n) = \varphi(n')$ as elements of $H$.
	
	In the first case as $(X,f)$ is an expansive $H$-action on a closed subset of the totally disconnected space $\{0,1\}^{\NN}$, it is conjugated to a subshift $Z \subset \ag^H$. Furthermore, using Lemma~\ref{lemma_symbolicfactoriseffective} we obtain that $Z$ is an effectively closed $H$-subshift. In the second case we just work with $(Z,\sigma)$ from the beginning. The argument is the same for both cases.

	For simplicity, let's first suppose $\ag = \{0,1\}$ and tackle the general case later. Consider the map $\rho : Z \to \{0,1\}^{\NN}$ where $\rho(z)_n = z_{\varphi(n)}$ where $\varphi(n) \in S^*$ is identified as an element of $H$. Consider the set $\Omega = \rho(Z)$ and the $H$-action $f' : H \times \Omega \to \Omega$ defined as $f'_h(\rho(z)) = \rho(\sigma_h(z))$. Clearly $\rho$ is a conjugacy between $(Z,\sigma)$ and $(\Omega,f')$. We claim that $(\Omega,f')$ is an effectively closed $H$-flow.
	
	Indeed, let $w \in \{0,1\}^*$. A Turing machine which accepts $w$ if and only if $[w] \in \{0,1\}^{\NN} \setminus \Omega$ is given by the following scheme: for each pair $(n,n')$ in the support of $w$ run $T$ in parallel. if $T$ accepts for a pair such that $w_n \neq w_{n'}$ then accept $w$ (this means that $w$ did not codify a configuration in $\ag^{\ZZ}$ as two words codifying different group elements have different symbols). Also, in parallel, use the algorithm recognizing a maximal set of forbidden patterns for $Z$ over the pattern coding $c_w = (\varphi(n),w_n)_{n \leq |w|}$. This eliminates all $w$ which codify configurations containing forbidden patterns in $Z$. For $f_s^{-1}[w]$ just note that the application $n \to \varphi(s^{-1}\varphi^{-1}(n))$ is recursive, thus $f_s^{-1}[w]$ can be calculated.
	
	As $(\Omega,f')$ is effectively closed, using Theorem~\ref{simulationTHEOREM} we can construct the sofic subshift $\texttt{Final}(\Omega,f')$. Instead of applying $\Upsilon \circ \pi$ to the $H$-subaction consider the local function $\Phi : \{0,1,\$\}^3 \to \{0,1\}$ defined as follows: Let $u \in \{0,1,\$\}^3$. If $u$ contains the word $0\$$, then $\Phi(u) = 0$, otherwise $\Phi(u) = 1$. Notice that $\phi : \{0,1,\$\}^{\ZZ} \to \{0,1\}^{\ZZ}$ defined by $\phi(y)_n = \Phi(\sigma_{-n}(y))$ satisfies that for any $x \in \{0,1\}^{\NN}$ then $\phi(\Psi_1(x)) = (x_0)^{\ZZ}$. Indeed, every substring of size $3$ is a cyclic permutation of $ax_0\$$ for some $a \in \{0,1,\$\}$.
	
	Now extend $\Phi$ to the support $F = ((\{0,1,2\},0),1_H)$ and let $\phi: \{0,1\$\}^G \to \{0,1\}^G$ be defined by $\phi(y)_{(z,h)} = \Phi(\sigma_{(z,h)^{-1}}(y)|_{F})$ and recall that $\texttt{Layer}^H_{1,1_H}$ is the projection to the layer containing in each horizontal strip a codification of $\Psi_1(x)$ for some $x \in \Omega$. Therefore, we can obtain a sofic $G$-subshift $$\texttt{Proj}(\Omega,f') = \phi(\texttt{Layer}^H_{1,1_H}(\texttt{Final}(\Omega,f')))$$
	
	We claim that $\pi_H(\texttt{Proj}(\Omega,f')) = (Z,\sigma)$. Note that $\forall y \in \texttt{Proj}(\Omega,f')$ then $y_{(z,h)} = (f'_{h^{-1}}(x))_0$. This means the $H$-projective subdynamics of $y$ is the configuration $\widetilde{x} \in \ag^H$ where $\widetilde{x}_h = (f'_{h^{-1}}(x))_0$. We claim $\widetilde{x} \in Z$. Indeed, as $x \in \Omega$ then there exists $z \in Z$ such that $x = \rho(z)$. By definition of $\rho$ and the fact that $\varphi(0) = \epsilon$ we have $$\widetilde{x}_h = (f'_{h^{-1}}(x))_0 = (\rho(\sigma_{h^{-1}}(z)))_0 = (\sigma_{h^{-1}}(z))_{1_H} = z_{h}.$$ 
	Conversely, every $z \in Z$ is realized by some $\rho(z) \in \Omega$, and there exists $y \in \texttt{Final}(\Omega,f')$ such that $\Upsilon(\pi(y))= \rho(z)$. Thus $z_h = \phi(\texttt{Layer}^H_{1,1_H}(y))_{(0,h)}$ henceforth $z \in \pi_H(\texttt{Proj}(\Omega,f')).$
	
	Therefore the $H$-projective subdynamics of $\texttt{Proj}(\Omega,f')$ is $(Z,\sigma)$ which is conjugated to $(X,f)$.
	
	In the case of a bigger $\ag$, we can code each $a \in \ag$ as a word in $\{0,1\}^k$ and redefine $\rho$ such that for $z\in Z$ then $\rho(z)_{n} = (z_{\varphi( \lfloor n/k \rfloor )})_{n\mod{k}}$. Everything stays the same in the construction except that the factor $\phi$ has support $\{0,\dots,3^k-1\}$ instead of $\{0,1,2\}$. it recovers from $\Psi_1(x)$ the first values $x_0,\dots,x_{k-1}$ and uses them to decode the corresponding $a \in \ag$.\end{proof}

\subsection{Existence of strongly aperiodic SFT in a class of groups obtained by semidirect products}

Next we show how these previous theorems can be applied to produce strongly aperiodic subshifts of finite type. We say a $G$-subshift $(X,\sigma)$ is \emph{strongly aperiodic} if the shift action is free, that is, $\forall x \in X, \sigma_g(x) = x \implies g = 1_G$.

\begin{theorem}\label{theoremaperiodic}
	Let $H$ be a finitely generated group and $(X,f)$ a non-empty effectively closed $H$-flow which is free. Then $G \cong \ZZ^2 \rtimes H$ admits a non-empty strongly aperiodic SFT.
\end{theorem}

\begin{proof}
	We begin by recalling the following general property of factors. Suppose there is a factor $\phi: (X,f) \twoheadrightarrow (Y,f')$. and let $x \in X$ such that $f_g(x) = x$. Then $f'_g(\phi(x)) = \phi(f_g(x)) = \phi(x) \in Y$. This means that if $f'$ is a free action then $f$ is also a free action.
	
	By Theorem~\ref{simulationTHEOREM} we can construct the $(\ZZ^2 \rtimes H)$-SFT $\widehat{\texttt{Final}}(X,f)$ such that $(\widehat{\texttt{Final}}(X,f),\sigma_H)$ is an extension of $(X,f)$ via the factor $\phi_1 = \Upsilon \circ \pi \circ \phi$. We also consider the factor $\phi_2 = \texttt{Sub}_{(1,1)} \circ \phi$ which sends $\widehat{\texttt{Final}}(X,f)$ first to $\texttt{Final}(X,f)$ and then to its $\texttt{Sub}_{(1,1)}$ layer.
	
	Let $y \in \widehat{\texttt{Final}}(X,f)$ and $(z,h) \in \ZZ^2 \rtimes H$ such that $\sigma_{(z,h)}(y) = y$. This implies that $\phi_2(y) = \sigma_{(z,h)}(\phi_2(y)) = \sigma_{(z,1_H)}(\sigma_{(0,h)}(\phi_2(y)))$. As we have seen in the proof of Theorem~\ref{simulationTHEOREM}, the action $\sigma_{(0,h)}$ leaves the lattices $(B_m)_{m \in \NN}$ of $\texttt{Sub}_{(1,1)}$ invariant in the $(\ZZ^2, 1_H)$-coset. Let $M > ||z||^2$. Then $\sigma_{(z,0)}$ does not leave invariant the lattice $B_m$. This implies that $z = \vec{0}$.
	Therefore, $\sigma_{(\vec{0},h)}(y) = y$. Applying $\phi_1$ we obtain that $f_h(y) = y$, and thus $h = 1_H$. Therefore $(z,h) = (\vec{0},1_H)$ and $\widehat{\texttt{Final}}(X,f)$ is strongly aperiodic. It is non-empty as $X \neq \emptyset$. \end{proof}

Theorem~\ref{theoremaperiodic} allows us to produce strongly aperiodic subshifts in many groups, we state this in a corollary.

\begin{corollary}\label{corollaryaperioci}
	Let $H$ be a finitely generated group with decidable word problem, then $\ZZ^2 \rtimes H$ admits a non-empty strongly aperiodic SFT.
\end{corollary}

\begin{proof}
	In \cite{AubBarTho2015} an effectively closed $H$-subshift is constructed for all finitely generated groups $H$ with decidable word problem. One way to construct this object is as follows: a Theorem~\cite{alonetal_nonrepetitivecoloringofgraphs} of Alon, Grytczuk, Haluszczak and Riordan uses Lov\'asz local lemma to show that every finite regular graph of degree $\Delta$ can be vertex-colored with at most $(2e^{16}+1)\Delta^2$ colors such that the sequence of colors in any non-intersecting path does not contain a square. Using compactness arguments this can be extended to Cayley graphs $\Gamma(H,S)$ of finitely generated groups where the bound takes the form $2^{19}|S|^2$ colors where $|S|$ is the cardinality of a set of generators of $H$. One can also show that the set of square-free vertex-coloring of $\Gamma(H,S)$ yields a strongly aperiodic subshift. In the case where $G$ has decidable word problem, a Turing machine can construct the sequence of balls $B(1_G,n)$ and enumerate a codification of all patterns containing a square colored path. Therefore we obtain an effectively closed, strongly aperiodic and non-empty subshift. Using the fact that $G$ is recursively presented one can do the coding of theorem $\ref{theorem.projective}$ to obtain a free non-empty effectively closed $H$-flow $(\Omega,f')$. Applying Theorem~\ref{theoremaperiodic} concludes the proof.\end{proof}

We remark that this corollary is an alternative proof to a construction done by Ugarcovici, Sahin and Schraudner in 2014 showing that the discrete Heisenberg group $\mathcal{H}$ admits non-empty strongly aperiodic SFTs. This falls directly from our theorem as $\mathcal{H} \cong \ZZ^2 \rtimes_{\varphi} \ZZ$ for $\varphi(1) = \begin{pmatrix}
1 & 1 \\ 0 & 1
\end{pmatrix}$. In their proof they use a similar trick using as a base the Robinson tiling~\cite{Robinson1971}. They use the lattices of crosses in this object to match the different $(\ZZ^2,0)$-cosets correctly to force a trivial action in the $\ZZ$ direction and use a counter machine to create aperiodicity in the other direction. In our construction the Robinson tiling got replaced by the substitutive subshifts $\texttt{Sub}_{(a,b)}$ which are able to match correctly the cosets of any possible automorphism and the counter machine by the simulation of the free $H$-flow. We also remark that Corollary~\ref{corollaryaperioci} answers some open questions in their talk asking the same property for the Flip, Sol groups and the powers of the Heisenberg group. The only case it does not solve is the one of their two-dimensional Baumslag Solitar group, as the matrix they used to define it is not invertible and can not be expressed as a semidirect product.

A theorem of Jeandel~\cite{Jeandel2015} says that for recursively presented groups $G$, the existence of a non-empty strongly aperiodic subshift $X \subset \ag^G$ implies that the word problem of $G$ is decidable. We can extend this to the case of arbitrary flows. This gives a deep relation between computability and dynamical properties.

\begin{corollary}
	Let $H$ be a recursively presented and finitely generated group. There exists a strongly aperiodic effectively closed $H$-flow if and only if the word problem of $H$ is decidable.
\end{corollary}

\begin{proof}
	If the word problem of $H$ is decidable, we can use the effectively closed subshift constructed in~\cite{AubBarTho2015} as an example. Conversely, Jeandel's result implies that if a recursively presented group admits a non-empty effectively closed and strongly aperiodic subshift then it's word problem is decidable. Using Theorem~\ref{theoremaperiodic} we can construct a strongly aperiodic subshift from any free effectively closed $H$-flow. Therefore the word problem of $H$ is decidable.
\end{proof}

\subsection{A generalization and comments on the size of the extension}

In this last portion we want to make explicit that the technique used in the proof of Theorem~\ref{simulationTHEOREM} can be easily be generalized to the following context

\begin{theorem}\label{simulationTHEOREMZD}
	Let $H$ be finitely generated group, $d \geq 2$ and $G = \ZZ^d \rtimes H$. For every $H$-effectively closed flow $(X,f)$ there exists a $G$-SFT whose $H$-subaction is an extension of $(X,f)$.
\end{theorem}

Indeed, instead of considering vectors in $(\ZZ/3\ZZ)^2 \setminus \{ \vec{0} \}$ we use $v \in (\ZZ/3\ZZ)^d\setminus \{ \vec{0} \}$ and $d$-dimensional substitutions $\texttt{s}_v$ defined analogously. The subshifts generated by these substitutions carry $\ZZ^d$-lattices and the configurations $z \in \texttt{Sub}_v$ can be described in the same way as before by lattices $B_m(z)$. The Toeplitz construction $\texttt{Top}(X,f)$ stays the same but instead of just constructing $\texttt{Top}(X,f)^{H}$ and $\texttt{Top}(X,f)^{V}$ we construct $\texttt{Top}(X,f)^{e_i}$ for every canonical vector $\{e_i\}_{i \leq d}$ where the $\langle e_i \rangle$-projective subdynamics yields $\texttt{Top}(X,f)$ and the configurations are extended periodically everywhere else. The rest of the construction translates directly to this setting.

We also want to remark the following: Hochman's theorem gives further information about the extension. Formally, given an effectively closed $\ZZ^d$-flow. A $\ZZ^{d+2}$-SFT is constructed such that the $\ZZ^d$-subaction is an \emph{almost trivial isometric extension} (ATIE) of the $\ZZ^d$-flow. An extension $(Z,f_Z) \twoheadrightarrow (Y,f_Y)$ is an ATIE if we can interpolate a factor $$(Z,f_Z)\twoheadrightarrow (Y,f_Y) \times (W,f_W) \twoheadrightarrow (Y,f_Y)$$ such that $(W,f_W)$ is an isometric action of a totally disconnected space, $(Y,f_Y) \times (W,f_W) \twoheadrightarrow (Y,f_Y)$ is the projection of the first coordinate  $(Z,f_Z)\twoheadrightarrow (Y,f_Y) \times (W,f_W)$ is almost everywhere $1-1$, that is, it satisfies that the set of points with unique preimage has measure 1 under any invariant Borel probability measure.

The idea behind the notion of ATIE is of an extension which is in a certain sense ``small''. It consists basically on adding a simple system $(W,f_W)$ as a product and then considering a measure equivalent action as the extension. Many properties such as the topological entropy (at least for $\ZZ^d$-actions) are preserved by taking ATIEs.

In our construction the only obstruction towards obtaining an ATIE is the use of the simulation theorem of effectively closed $\ZZ$-subshifts as projective subactions of sofic $\ZZ^2$-subshifts. This theorem in its current state does not yield an almost everywhere $1-1$ extension. The rest of the proof can be adapted to obtain an ATIE, for instance, the substitutive layers can be coupled in a single substitution to avoid the degree of freedom when either $a$ or $b$ are zero. Furthermore, the substitutive layers and the Toeplitz structure can be factorized in the isometric action as they are invariant under the $H$-subaction. Therefore, the maps $\Upsilon \circ \pi$ do not pose obstructions to obtaining an ATIE. Everything that remains is the factor $\phi: \widehat{\texttt{Final}}(X,f) \twoheadrightarrow \texttt{Final}(X,f)$. Here the substitutive layers don't present a problem as they come from a primitive substitution with unique derivation and thus Mozes's theorem~\cite{mozes} gives the almost $1-1$ SFT extension. The only thing that remains is the aforementioned almost $1-1$ SFT extension for $\texttt{Top}(X,f)^{H}$ and $\texttt{Top}(X,f)^{V}$ that could be obtained by refining that simulation theorem.

\bibliographystyle{plain}
\bibliography{Bibliography}
 
\end{document}